\documentclass[12pt]{article}

\usepackage{amsmath,amssymb,amsthm}
\usepackage{graphicx,ifpdf,tikz,color}
\usepackage{enumerate}

\vspace{8mm}
\title{Comparison of Wiener index and Zagreb eccentricity indices
\thanks{
 Email addresses: kexxu1221@126.com(K. Xu), kinkardas2003@gmail.com(K. C. Das), \newline sandi.klavzar@fmf.uni-lj.si (S.\ Klav\v{z}ar), 1213966965@qq.com(H. Li).
}
 }
 \author{Kexiang Xu $^{a},$ Kinkar Chandra Das $^{b},$ Sandi Klav\v{z}ar $^{c,d,e},$ Huimin Li $\/^{a}$ \\\\
 $^{a}$ \small  College of Science, Nanjing University of
 Aeronautics \& Astronautics,\\
 \small Nanjing, Jiangsu 210016, PR China\\
 $^{b}$ \small Department of Mathematics, Sungkyunkwan University,\\
 \small Suwon 440-746, Republic of Korea\\
  $^{c}$ \small Faculty of Mathematics and Physics, University of Ljubljana, Slovenia \\
 $^{d}$ \small Faculty of Natural Sciences and Mathematics, University of Maribor, Slovenia \\
  $^{e}$ \small Institute of Mathematics, Physics and Mechanics, Ljubljana, Slovenia
  }
 \date{}

\begin{document}

\newtheorem{theorem}{\bf Theorem}[section]
\newtheorem{corollary}[theorem]{\bf Corollary}
\newtheorem{algorithm}[theorem]{\bf Algorithm}
\newtheorem{lemma}[theorem]{\bf Lemma}
\newtheorem{claim}[theorem]{\bf Claim}
\newtheorem{observation}[theorem]{\bf Observation}
\newtheorem{proposition}[theorem]{\bf Proposition}
\newtheorem{conjecture}[theorem]{\bf Conjecture}
\newtheorem{remark}[theorem]{\bf Remark}
\newtheorem{problem}[theorem]{\bf Problem}
\newtheorem{example}[theorem]{\bf Example}
\newtheorem{definition}[theorem]{\bf Definition}

\renewcommand{\proofname}{\textup{\textbf{Proof}}}
\newcommand{\s}{\color{blue}}
\newcommand{\avd}{{\rm avd}}
\newcommand{\avt}{{\rm avt}}
\newcommand{\Tr}{{\rm Tr}}
\newcommand{\diam}{{\rm diam}}
\newcommand{\rad}{{\rm rad}}
\newcommand{\cp}{\,\square\,}
\newcommand{\Ecc}{{\rm Ecc}}

\maketitle

\begin{center}
(Received on December 12, 2019)
\end{center}

 \begin{abstract}
The first and the second Zagreb eccentricity index of a graph $G$ are defined as $E_1(G)=\sum_{v\in
V(G)}\varepsilon_{G}(v)^{2}$ and $E_2(G)=\sum_{uv\in E(G)}\varepsilon_{G}(u)\varepsilon_{G}(v)$, respectively, where $\varepsilon_G(v)$ is the eccentricity of a vertex $v$. In this paper the invariants $E_1$, $E_2$, and the Wiener index are compared on graphs with diameter $2$, on trees, on a newly introduced class of universally diametrical graphs, and on Cartesian product graphs. In particular, if the diameter of a tree $T$ is not too big, then $W(T) \ge E_2(T)$ holds, and if the diameter of $T$ is large, then $W(T) <  E_1(T)$ holds.
\end{abstract}



 \baselineskip=0.30in

 \section{Introduction}

Graphs considered in this paper are finite, undirected, and simple. If $G=(V(G), E(G))$ is a graph, we will use $n(G) = |V(G)|$ for its order and $m(G) = |E(G)|$ for its size. The \textit{degree} $\deg_G(v)$ of $v\in V(G)$ is the number of vertices in $G$ adjacent to $v$. The complement of $G$ is denoted with $\overline{G}$. The eccentricity $\varepsilon_G(v)$ (or $\varepsilon(v)$ for short) of a vertex $v\in V(G)$ is the maximum distance from $v$ to the vertices of $G$, that is, $\varepsilon_G(v)=\max\limits_{u\in V(G)} d_G(v,u)$. The \textit{eccentric set} of $v$ is $\Ecc_G(v) = \{u:\ d_G(v,u) = \varepsilon_G(v)\}$, cf.~\cite{XLDK2018}, and the \textit{total eccentricity} of $G$ is $\varepsilon(G) = \sum\limits_{v\in V(G)}\varepsilon_G(v)$ (see more related results in \cite{I2012}). The \textit{diameter} and the \textit{radius} of $G$ are $\diam(G)=\max\limits_{v\in V(G)}\varepsilon_G(v)$ and $\rad(G)=\min\limits_{v\in V(G)}\varepsilon_G(v)$, respectively. A graph is \textit{$k$-self-centered graph} if $\diam(G) = \rad(G) = k$.

A graphical invariant is a function from the set of graphs to the reals which is invariant under graph automorphisms. In chemical graph theory, graphical invariants are most often referred to as topological indices. Among the oldest topological indices are the well-known Zagreb indices first introduced in~\cite{gutmantri1972}, where Gutman and Trinajsti\'{c} examined the dependence of total $\pi$-electron energy on molecular structure. The work was further elaborated in~\cite{grt1975}. The {\em first Zagreb index} $M_{1}(G)$ and the {\em second Zagreb index} $M_{2}(G)$ of a (molecular) graph are defined as
$$M_{1}(G)=\sum\limits_{v\in
V(G)}\deg_{G}(v)^{2}\quad {\rm and}\quad M_{2}(G)=\sum\limits_{uv\in
E(G)}\deg_{G}(u)\deg_{G}(v)\,.$$
These two classical topological indices reflect the extent of
branching of the molecular carbon-atom skeleton
\cite{tc2000}. See~\cite{BDFG2017, javaid-2019, ji-2018, XDB2014, yurtas-2019, zhan-2019} for various recent results on Zagreb indices. In analogy with the first and the second Zagreb index, Vuki\v{c}evi\'{c} and Graovac~\cite{VG2010} introduced the {\em first} and the {\em second Zagreb eccentricity index} as
$$E_1(G)=\sum\limits_{v\in
V(G)}\varepsilon_{G}(v)^{2}\quad {\rm and}\quad E_2(G) = \sum\limits_{uv\in E(G)}\varepsilon_{G}(u)\varepsilon_{G}(v)\,.$$
For  properties of $E_1$ and $E_2$ see~\cite{DLG2013, DZT2012, QZL2017, XZT2011}, see also~\cite{li-2017,XDL2016,YQTF2014} for another role of eccentricity in chemical graph theory.

The oldest topological index in chemical graph theory, however, is the Wiener index~\cite{Wi1947}. It is still of very high current nterest, cf.~\cite{AK2018, dobrynin-2018, knor-2016, knor-2019, wei-2019, zhang-2019} and is defined on a connected graph $G$ as $W(G)=\sum\limits_{\{u,v\}\subseteq V(G)}d_G(u,v)$. For a vertex $v\in V(G)$, the {\em transmission} $\Tr_G(v)$ of $v$ is the sum of the distances from $v$ to other vertices in $G$, so that
\begin{eqnarray}
W(G)=\frac{1}{2}\sum\limits_{v\in V(G)}\Tr_G(v)\,.\label{e1.2}
\end{eqnarray}

Recently, some results were proved on the comparison between the Wiener index and the total eccentricity of graphs~\cite{DAKD2019}, while in~\cite{XK19} the so-called Wiener complexity was compared with the eccentric complexity. In this paper we continue the research in this direction by comparing the Wiener index,  the first Zagreb eccentricity index, and the second Zagreb eccentricity index. In the next section we focus on graphs with diameter $2$ and prove that in the majority of cases either $E_1(G) < E_2(G)$ or $E_1(G) > E_2(G)$ holds for such graphs $G$, and classify when one of the two options occurs. In Section~\ref{sec:trees} we consider trees, while in Section~\ref{sec:UD} we introduce and study universally diametrical graphs. We conclude with two results on Cartesian product graphs.

\section{Graphs with diameter $2$}

$K_n$ is the unique graph of order $n$ and diameter $1$. Clearly, $E_2(K_n)=W(K_n)={n\choose 2}>n=E_1(K_n)$ for $n\geq 3$.  Hereafter we thus consider the graphs with diameter at least $2$, in this section those with diameter $2$. If $n\ge 3$, then denote by ${\cal{G}}_n^2$ the set of graphs of order $n$ with diameter $2$. We first compare $E_1$ and $E_2$.

\begin{proposition}\label{k2-SC}
If $G$ is a self-centered, not-complete graph, then $E_2(G)\geq E_1(G)$ with equality holding if and only if $G$ is a cycle.
\end{proposition}

\begin{proof}
Set  $m = m(G)=m$ and $n = n(G)$.  Clearly,  $\delta(G)\geq 2$ because a pendant vertex has different eccentricity than its support vertex. Hence
    $$2m=\sum\limits_{v\in V(G)}\deg_G(v)\geq 2n,$$
that is, $m\geq n$. If $m=n$, then $G\cong C_n$ in which case $E_2(G)=n\,\lfloor\frac{n}{2}\rfloor^2=E_1(G)$. Otherwise, $m>n$ and hence $E_2(G)=m\cdot \varepsilon(G)^2>n\cdot \varepsilon(G)^2=E_1(G)$.
\end{proof}

For graphs with diameter $2$, Proposition~\ref{k2-SC} immediately implies:

\begin{corollary}\label{cor:2-SC}
If $G$ is a self-centered graph with $\diam(G) = 2$, then $E_2(G)\ge E_1(G)$. Moreover, equality holds if and only if $G\in \{C_4, C_5\}$.
\end{corollary}

To formulate the next result, we need some preparation. A vertex $v\in V(G)$ is a \textit{universal vertex} if $\deg_G(v)=n(G)-1$. We will denote with $n'(G)$ the number of universal vertices of $G$ and with $G'$ the subgraph of $G$ induced by the non-universal vertices. In other words, $G'$ is obtained from $G$ by removing all of its universal vertices. Finally, denote by $\avd(G)$ the \textit{average degree} of graph $G$, that is, $\avd(G)=\frac{2m(G)}{n(G)}$. Then we have:

\begin{theorem}\label{thm:2-non-SC}
Let $G$ be a non-self-centered graph with $n(G)\ge 3$ and $\diam(G) = 2$. If (i) $n'(G)\geq 3$, or (ii)  $n'(G) = 2$ and $\avd(G') > 0$, or (iii) $n'(G) = 1$ and $\avd(G') > 1+\frac{1}{2(n-1)}$, then $E_1(G) < E_2(G)$. Otherwise, $E_1(G) > E_2(G)$.
\end{theorem}

\begin{proof}
Set $n = n(G)$, $m = m(G)$, and $n' = n'(G)$. Then $m={n'\choose 2}+n'(n-n')+x$, where $x=m(G')$. Consequently, $E_1(G)=4(n-n') + n' = 4n - 3n'$ and $E_2(G) =  {n'\choose 2}+2n'(n-n')+4x$. Then it follows that
\begin{eqnarray}E_2(G)-E_1(G)&=&2(n'-2)(n-n')+\frac{n'(n'-3)}{2}+4x.
\label{neweq1}
\end{eqnarray}

Since $G$ is a non-self-centered graph with $\diam(G) = 2$,  we have $n'\ge 1$. We distinguish the following three cases on the value of $n'$.

Suppose first that $n^{\prime}\geq 3$. Then by~\eqref{neweq1} we get that $E_2(G)-E_1(G) \geq 2(n-n')>0$, where the last inequality holds because $G$ is not complete and thus $n > n'$.

Suppose next that $n'=2$. Using~ \eqref{neweq1} we obtain that $E_2(G)-E_1(G)=4x-1$ and therefore $E_2(G)>E_1(G)$ provided that  $\avd(G')=\frac{2x}{n-2}>0$. Otherwise we have $E_2(G)<E_1(G)$.

Assume next that $n' = 1$. Applying~\eqref{neweq1} again, we get $E_2(G)-E_1(G) = 4x-2n+1$ and (since $n'=1$) also $\avd(G')=\frac{2x}{n-1}$. Therefore $E_2(G)-E_1(G)>0$ if $\avd(G') > 1+\frac{1}{2(n-1)}$. Otherwise, we have $\avd(G')\leq 1+\frac{1}{2(n-1)}$. We claim that $\avd(G')\neq 1+\frac{1}{2(n-1)}$. Indeed, if this would be the case, then we would derive the equality $4 m(G') = 2n - 1$, which is not possible. Clearly, we have $E_1(G)<E_2(G)$ if $\avd(G') < 1+\frac{1}{2(n-1)}$.
\end{proof}

If $G\in {\cal{G}}_n^2$, then $d_G(u,v)=2$ holds for each non-adjacent vertices $u$ and $v$, hence the following result holds immediately.

\begin{proposition}\label{d-2-W}
If $n\geq 3$ and $G\in {\cal{G}}_n^2$ has $m$ edges, then $W(G)=n(n-1)-m$.
\end{proposition}

Next we compare $W$ with $E_1$ and $E_2$ for the graphs from ${\cal{G}}_n^2$.

\begin{theorem}\label{universal}
If $n\geq 9$ and $G\in {\cal{G}}_n^2$, then  $W(G)>E_1(G)$.
\end{theorem}

\begin{proof}
Set $n' = n'(G)$ and $m = m(G)$. Then  $E_1(G)=4n-3n'$ and $m<\frac{n(n-1)}{2}$ since $G\ncong K_n$. So, by Proposition~\ref{d-2-W}, we have \begin{eqnarray*}W(G)-E_1(G)&=&n(n-5)-m+3n'\\
&>&\frac{n(n-9)}{2}+3n'\\
&\geq&0\,,
\end{eqnarray*}
the last inequality holding by the assumption $n\ge 9$.
\end{proof}

Since $E_1(G)=4n(G)$ and $E_2(G)=4m(G)$ hold for a self-centered graph $G$ with diameter $2$, Proposition~\ref{d-2-W} yields:

\begin{proposition}\label{prop:2-SC} If $G$ is a self-centered graph  of order $n$, size $m$, and diameter $2$, then the following statements hold.
\begin{enumerate}[(i)]
\item  $W(G)>E_1(G)$ if and only if  $m<n(n-5)$.
\item $W(G)>E_2(G)$ if and only if  $m<\frac{n(n-1)}{5}$.
\end{enumerate}
 \end{proposition}

In the following we consider non-self-centered graphs $G\in{\cal{G}}_n^2$.

\begin{theorem}\label{E2-W}
If $n\ge 3$ and $G\in{\cal{G}}_n^2$ with $n'(G)>\frac{n-1}{2}$, then $E_2(G)>W(G)$.
\end{theorem}

\begin{proof}
Set $n = n(G)$, $m = m(G)$, and $n' = n'(G)$. Since $n' > \frac{n-1}{2}$, $G$ is non-self-centered. As already observed in the proof of Theorem~\ref{thm:2-non-SC}, $m = n'(n-n')+x+{n'\choose 2}$, where $x = m(G')$. Then $E_2(G)={n'\choose 2}+2n'(n-n')+4x$.    Moreover, $W(G)=n(n-1)-n'(n-n')-x-{n'\choose 2}$ by Proposition~\ref{d-2-W}. Then it follows that \begin{eqnarray*}E_2(G)-W(G)&=&5x+2{n'\choose 2}+3n'(n-n')-n(n-1)\\
&=&5x-\Big[2n'^2-(3n-1)n'+n(n-1)\Big]\\
&=&5x-2(n'-n)\Big(n'-\frac{n-1}{2}\Big)\\
&>&0
\end{eqnarray*}
for $n'>\frac{n-1}{2}$, completing the argument.
\end{proof}

\begin{corollary} \label{E2-W-adv}
Let $G\in{\cal{G}}_n^2$ with $0<n'(G) \leq \frac{n-1}{2}$.
\begin{enumerate}[(i)]
\item If $\avd(G')>\frac{2}{5}(n-1-2n'(G))$, then $E_2(G)>W(G)$.
\item If $\avd(G')<\frac{2}{5}(n-1-2n'(G))$, then $E_2(G)<W(G)$.
\end{enumerate}
\end{corollary}

\begin{proof}
Set again $n' = n'(G)$ and $x = m(G')$. Using the argument from the proof of Theorem \ref{E2-W} we have
\begin{eqnarray*}
E_2(G)-W(G)&=&5x+2{n'\choose 2}+3n'(n-n')-n(n-1)\\
&=&5x-2(n-n')\Big(\frac{n-1}{2}-n'\Big)\\
&=&\frac{5}{2}(n-n')\Big[\frac{2x}{n-n'}-\frac{2}{5}(n-1-2n')\Big]\\
&>&0\,,
\end{eqnarray*}
where the last inequality follows by the assumption $\avd(G')=\frac{2x}{n-n'}>\frac{2}{5}(n-1-2n')$.

The above argument works also if $\avd(G')<\frac{2}{5}(n-1-2n'(G))$, the difference being only in the last estimate which becomes less than $0$.
\end{proof}

We conclude the section with the following construction.

\begin{theorem} \label{E2-W-n1}
For each integer $n'\in (0,n-2]$, there exists a graph $G\in{\cal{G}}_n^2$ with $n'(G) = n'$ such that  $E_2(G)>W(G)$.\end{theorem}

\begin{proof}
If $n'>\frac{n-1}{2}$, the result holds by Theorem~\ref{E2-W}, hence it remains to consider the cases $n'\in (0,\frac{n-1}{2}]$. Let $G\in{\cal{G}}_n^2$ and let $V'$ be the set of non-universal vertices in $G$, so that $G'$ is the subgraph of $G$ induced by $V'$. If $\avd(G')>\frac{2}{5}(n-1-2n')$, then $E_2(G)>W(G)$ from Corollary \ref{E2-W-adv}. Otherwise, $\avd(G')\leq \frac{2}{5}(n-1-2n')$.  Let $V_0^{\prime}=\{v:v\in V_0, \deg_{G'}(v)<n-n'-2\}$.  Note that $\frac{2}{5}(n-1-2n')<n-n'-2$ for $n>5$. Then $\emptyset\subset V_0^{\prime}\subseteq V'$. Now we construct a graph $G^*$ obtained by inserting some edges among the vertices in $V_0^{\prime}$ such that $G^*[V_0]$ is a graph obtained by removing $i\leq\lfloor\frac{n-n^{\prime}}{2}\rfloor$ independent edges from $K_{n-n^{\prime}}$. Then $G^*\in {\cal{G}}_n^2$ with $\avd(G^*[V_0])\geq n-n'-2>\frac{2}{5}(n-1-2n')$. The result then follows from Corollary~\ref{E2-W-adv}. \end{proof}

\section{Trees}
\label{sec:trees}

In this section we compare $W$ with $E_1$ and with $E_2$ on the class of trees. The main results assert that if the diameter of a tree is not too big, then $W\ge E_2$ and if the diameter of a tree is large, then $W <  E_1$.

\begin{theorem}\label{E2lWt}
If $T$ is a tree with $n(T)\geq 3$ and $\diam(T) \leq \frac{1+\sqrt{4n-3}}{2}$, then $E_2(T)\leq W(T)$ with equality holding if and only if $T\cong P_3$.
\end{theorem}

\begin{proof}
Set $n = n(T)$ and $d = \diam(T)$. Clearly, $\varepsilon_T(v)\varepsilon_T(u)\leq d(d-1)$ holds for an edge $uv\in E(T)$ with equality holding if and only if one of the vertices $u$ and $v$ is diametrical.  Since $d\leq \frac{1+\sqrt{4n-3}}{2}$, we have $d(d-1)\leq n-1$. For an edge $uv\in E(T)$ let $n_u$ and $n_v$ be the number of vertices closer to $u$ than to $v$, and closer to $v$ than to $u$, respectively. Clearly, $n_u+n_v=n$. Recall further the well-known fact going back to Wiener~\cite{Wi1947} that $W(T) = \sum\limits_{uv\in E(T)} n_un_v$. Hence for any edge $uv\in E(T)$ we have
$$\varepsilon_T(v)\varepsilon_T(u) \leq d(d-1) \leq n-1 \leq n_un_v\,,$$
which after summing over all the edges of $T$ yields $E_2(T)\leq W(T)$.  Moreover, the equality holds if and only all three equalities above hold for each edge $uv\in E(T)$. Equivalently, each edge $uv\in E(T)$ is a pendant edge in $T$, and $n-1=d(d-1)$. Only the path $P_3$ of order $3$ has these properties.
\end{proof}

We have thus seen that if the diameter of a tree is relative small, then $W\ge E_2$. On the other hand,  if  the diameter of a tree is large, then $W <  E_1$:

\begin{theorem}\label{tree-E1-W}
If $T$ is a tree with  $n(T)>3$ and $\diam(T)\geq \frac{2n}{3}$, then $W(T)<E_1(T)$.
\end{theorem}

\begin{proof}
Set $n = n(T)$, $d = \diam(T)$ and $r = \rad(T)$. Assume that $d$ is even. (The proof for the  case when $d$ is odd is analogous and hence omitted.) Then  $T$ has radius $r=\frac{d}{2}$ and $d>2$ holds because $d\geq \frac{2n}{3}$ and $n>3$.  From definitions, it suffices to prove that $\frac{\Tr_T(v)}{2}\leq \varepsilon_T(v)^2$ holds for  each vertex $v$ of $T$, and that for at least one vertex strict inequality holds. Let $P$ be a diametrical path in $T$ with $y,z$ as two diametrical vertices. Then  $\varepsilon_T(v)=\max\{d_T(v,y),d_T(v,z)\}$ for any vertex $v\in V(T)$. Next we bound the value of $\frac{\Tr_T(v)}{2}$ for vertices $v$ of $T$ and distinguish three cases.

Suppose first that $v$ is a diametrical vertex in $T$. Then
\begin{eqnarray*}\frac{\Tr_T(v)}{2}&\leq&\frac{1}{2}\Big[1+2+\cdots+d+(n-d-1)d\Big]\\
&=&\frac{1}{2}\Big(n-\frac{d+1}{2}\Big)d\\
&<&d^2 = \varepsilon_T(v)^2\,,
\end{eqnarray*}
where the strict inequality holds because $d\geq \frac{2n}{3}$.

Suppose next that $v$ is a central vertex in $T$. (Since $d$ is even, such a verttex is actually unique.) Then \begin{eqnarray*}
\frac{\Tr_T(v)}{2}&\leq&\frac{1}{2}\Big[2(1+2+\cdots+r)+(n-2r-1)r\Big]\\
&=&\frac{(n-r)r}{2}\\
&\le &r^2\,,
\end{eqnarray*}
where the last inequality holds since $d\geq \frac{2n}{3}$ and $d=2r$.

In the last case assume that $v$ is neither a diametrical nor the central vertex of $T$. Then $\varepsilon_T(v)=k$,  where $\frac{d+2}{2}\leq k\leq d-1$. In the first subcase assume that $v$ lies on $P$. Then
\begin{eqnarray*}
\frac{\Tr_T(v)}{2}&\leq&\frac{1}{2}\Big[1+2+\cdots+k+1+2+\cdots+d-k+(n-d-1)k\Big]\\
&=&\frac{1}{2}\Big[(n-d+\frac{k-1}{2})k+\frac{(d-k+1)(d-k)}{2}\Big]\\
&=&\frac{1}{2}\Big[(n-2d+k-1)k+\frac{d^2+d}{2}\Big].
\end{eqnarray*}
Thus it follows that
\begin{eqnarray*}
\varepsilon_T(v)^2-\frac{\Tr_T(v)}{2}&\geq&k^2-\frac{1}{2}\Big[(n-2d+k-1)k+\frac{d^2+d}{2}\Big]\\
&=&\frac{1}{2}\Big[k^2+(1+2d-n)k-\frac{d^2+d}{2}\Big]\\
&\geq&\frac{1}{2}\Big[\frac{d+2}{2}(\frac{d+2}{2}+1+2d-n)-\frac{d^2+d}{2}\Big]\\
&\geq&\frac{3d+4}{4}>0
\end{eqnarray*}
for $k\geq \frac{d+2}{2}$ with $d\geq \frac{2n}{3}$, that is, $n\leq \frac{3d}{2}$.

In the second subcase assume that $v$ is not a vertex of $P$. Let $u$  be the vertex of $P$ closest to $v$. Clearly, $u\ne y,z$. Let $d_T(y,v)=k$. Then we get $d_T(u,y)=k-x$ and $d_T(u,z)=d-k+x\leq k-x$ which implies that $1\leq x\leq k-\frac{d}{2}$. Then
\begin{eqnarray*}
\Tr_T(v)&\leq&1+2+\cdots+k+x+1+\cdots+x+d-k+x+(n-d-1-x)k\\
&=&\frac{k(k+1)}{2}+\frac{(d-k+2x)(d-k+2x+1)}{2}-\frac{x(x+1)}{2}+(n-d-1-x)k\\
&=&k^2+\frac{d^2+d}{2}+2x(d-k)+\frac{3x^2+x}{2}+(n-2d-1-x)k\,,
\end{eqnarray*}
which gives
\begin{eqnarray}
2\varepsilon_T(v)^2-\Tr_T(v)&\geq&k^2+(2d+1-n)k-\frac{d^2+d}{2}-x(2d-3k)-\frac{3x^2+x}{2}\,.
\label{e3.1}
\end{eqnarray}
Consider the function
 $$h(x)=x(2d-3k)+\frac{3x^2+x}{2}$$
defined for $x\in [1,k-\frac{d}{2}]$. Then we have $h^{\prime}(x)=2d-3k+\frac{6x+1}{2}$ which implies that $h(x)$ is an increasing function on $x\geq k-\frac{2d}{3}-\frac{1}{6}$ and a decreasing function on $x\leq k-\frac{2d}{3}-\frac{1}{6}$.
Now we determine the maximum value of $h(x)$.

\noindent
${\bf Case\,1:}$ $k\geq \frac{2d}{3}+\frac{1}{6}$. In this case
    $$h(x)\leq \max\left\{h(1),h\Big(k-\frac{d}{2}\Big)\right\}.$$
 One can easily see that
  $$h(1)=2d-3k+2\leq \Big(k-\frac{d}{2}\Big)\Big(2d-3k+\frac{3\Big(k-\frac{d}{2}\Big)+1}{2}\Big)=h\Big(k-\frac{d}{2}\Big)$$
 as $k\geq \frac{d}{2}+1$. Thus we have
   $$h(x)\leq \Big(k-\frac{d}{2}\Big)\Big(2d-3k+\frac{3\Big(k-\frac{d}{2}\Big)+1}{2}\Big)=\Big(k-\frac{d}{2}\Big)\Big(\frac{5d}{4}-\frac{3k}{2}+\frac{1}{2}\Big).$$

\noindent
${\bf Case\,2:}$ $\frac{d}{2}+1\leq k<\frac{2d}{3}+\frac{1}{6}$. In this case we have
    $$h(x)\leq h\Big(k-\frac{d}{2}\Big)=\Big(k-\frac{d}{2}\Big)\Big(\frac{5d}{4}-\frac{3k}{2}+\frac{1}{2}\Big).$$

\noindent
From (\ref{e3.1}), we obtain
\begin{eqnarray}
2\varepsilon_T(v)^2-\Tr_T(v)&\geq&k^2+(2d+1-n)k-\frac{d^2+d}{2}-\Big(k-\frac{d}{2}\Big)\Big(\frac{5d}{4}-\frac{3k}{2}+\frac{1}{2}\Big)\nonumber\\
    &=&\frac{5k^2}{2}-nk+\frac{k}{2}+\frac{d^2}{8}-\frac{d}{4}.\label{kin-2}
\end{eqnarray}

\noindent
Note that $k\geq \frac{d}{2}+1\geq \frac{n}{3}+1$ as $d\geq \frac{2n}{3}$. Since $n\leq \frac{3d}{2}$, we have that $g(x)=\frac{5x^2}{2}-nx+\frac{x}{2}$ is a strictly increasing function on $x\geq \frac{d}{2}$. From (\ref{e3.1}), we have
    $$2\varepsilon_T(v)^2-\Tr_T(v)>g\left(\frac{d}{2}\right)+\frac{d^2}{8}-\frac{d}{4}=\frac{5d^2}{8}-\frac{nd}{2}+\frac{d^2}{8}=\frac{d(3d-2n)}{4}\geq 0,$$
 which implies $\varepsilon_T(v)^2>\frac{\Tr_T(v)}{2}$.
\end{proof}

We conclude the section with the following result.

\begin{theorem}\label{tree-com}
If $T$ is a tree with $n(T)>8$, then either $W(T)>E_1(T)$ or $W(\overline{T})>E_1(\overline{T})$.
\end{theorem}

\begin{proof}
Set $n = n(T)$ and $d = \diam(T)$. If $d=2$, then the assertion follows from Theorem~\ref{universal}. If $d=3$, then $T$ is a double star, where the two non-leaves of $T$ are adjacent to $n'$ and $n-2-n'$ leaves, respectively, where $1\leq n'\leq \lfloor\frac{n-2}{2}\rfloor$. It follows that
\begin{eqnarray*}
W(T)&=&n-1+2\Big[n'+n-2-n'+{n' \choose 2}+{n-n'-2\choose 2}\Big]+3n'(n-2-n')\\
&=&3n-5+n'(n'-1)+(n-2-n')(n-3-n')+3n'(n-2-n')\\
&=&(n-1)^2+(n-2)n'-n'^2\\
&\geq&(n-1)^2+(n-2)-1\\
&=&n^2-n+2\\
&>&9n-10=E_1(T)\,,
\end{eqnarray*}
that is, $W(T)>E_1(T)$ for $n>8$.

The last case to consider is when $d\geq 4$. From a well known fact that $\diam(\overline{G}) = 2$ if $\diam(G)\ge 3$  (see~\cite[Exercise 1.6.12]{bondymurty1976}), we have $\diam(\overline{T}) = 2$. Then $\overline{T}$ is a self-centered graph of order $n>8$ with $m(\overline{T})={n\choose 2}-(n-1)=\frac{(n-1)(n-2)}{2}$ and the assertion follows by Corollary~\ref{cor:2-SC}.
\end{proof}

\section{Universally diametrical graphs}
\label{sec:UD}

We say that a graph $G$ is \textit{universally diametrical} (UD for short) if  there exist diametrical vertices $u$ and $v$  of $G$, such that $\Ecc_G(w)\cap\{u,v\}\neq\emptyset$ for any vertex $w\in V(G)\setminus \{u,v\}$,  that is, at least one of $u$ and $v$ is eccentric to $w$. We further say that the vertices $u$ and $v$ form a \textit{universally diametrical pair} in $G$.  A universally diametrical graph $G$ is called a \textit{$k$-$(u,v)$-universally diametrical} (or $k$-$(u,v)$-UD for simplicity) graph if $d_G(u,v)=\diam(G)=k$.

Obviously, any tree is a UD graph. A sporadic example of a UD graph is shown  in Figure \ref{F0}.  Let further $A_k$, $k\ge 1$,  be the graph obtained by attaching $k$ pendant vertices to each of two diametrical vertices of $C_{4}$. Then $A_k$ is a $4$-UD graph for each $k\ge 1$. Note also that the $d$-dimensional hypercube $Q_d$ is a $d$-UD graph in which each pair of diametrical vertices form a universally diametrical pair.

\begin{figure}[ht!]
\begin{center}
\begin{tikzpicture}[scale=0.8,style=thick]
\def\vr{3pt}
\path (0,0) coordinate (a1);
\path (1,0) coordinate (a2);
\path (2,0) coordinate (a3);
\path (2,1) coordinate (b3);
\path (3,0) coordinate (a4);
\path (4,0) coordinate (a5);
\path (5,0) coordinate (a6);
\path (5,1) coordinate (b6);
\path (6,0) coordinate (a7);
\path (7,0) coordinate (a8);
\path (8,0) coordinate (a9);
\path (9,0) coordinate (a10);
\path (10,0) coordinate (a11);
\path (11,0) coordinate (a12);
\path (8,-1) coordinate (b9);
\path (10,1) coordinate (b11);

\draw (a1) -- (a2) -- (a3) -- (a4) -- (a5) --(a6) -- (a7) -- (a8) -- (a9) -- (a10) -- (a11) --(a12);
\draw (a2) -- (b3) -- (a4);
\draw (a3) -- (b3);
\draw (a5) -- (b6) -- (a7);
\draw (a6) -- (b6);
\draw (a8) -- (b9) -- (a10);
\draw (a9) -- (b9);
\draw (a11) -- (b11) -- (a12);
\draw (a1)  [fill=white] circle (\vr);
\draw (a2)  [fill=white] circle (\vr);
\draw (a3)  [fill=white] circle (\vr);
\draw (b3)  [fill=white] circle (\vr);
\draw (a4)  [fill=white] circle (\vr);
\draw (a5)  [fill=white] circle (\vr);
\draw (a6)  [fill=white] circle (\vr);
\draw (b6)  [fill=white] circle (\vr);
\draw (a7)  [fill=white] circle (\vr);
\draw (a8)  [fill=white] circle (\vr);
\draw (a9)  [fill=white] circle (\vr);
\draw (a10)  [fill=white] circle (\vr);
\draw (a11)  [fill=white] circle (\vr);
\draw (a12)  [fill=white] circle (\vr);
\draw (b9)  [fill=white] circle (\vr);
\draw (b11)  [fill=white] circle (\vr);
\draw (0,-0.5) node {$u$};
\draw (11,-0.5) node {$v$};

\end{tikzpicture}
\end{center}
\caption{ $11$-$(u,v)$-UD graph}
\label{F0}
\end{figure}
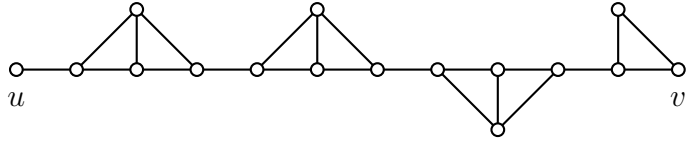

To prove the next first main result of this section, the following lemma will be useful.

\begin{lemma}\label{tr-ec}
Let $G$ be a connected graph with $v\in V(G)$. Then $\varepsilon(G)-\varepsilon_G(v)\geq \Tr_G(v)$ with equality holding if and only if $\varepsilon_G(u)=d_G(v,\,u)$ for any vertex $u\in V(G)\setminus \{v\}$.
\end{lemma}

\begin{proof}
From definitions, we have \begin{eqnarray*}
\varepsilon(G)-\varepsilon_G(v)&=&\sum\limits_{w\in V(G)\setminus\{v\}}\varepsilon_G(u)\\
              &\geq&\sum\limits_{u\in V(G)\setminus\{v\}}d_G(v,u)\\
              &=&\Tr_G(v)
 \end{eqnarray*} with equality holding if and only if $\varepsilon_G(w)=d_G(v,w)$ for any $u\in V(G)\setminus \{v\}$.
 \end{proof}

In the following, let $f(x)=2x^2+9x+6$ with $x>0$.

\begin{theorem}\label{uni-pend}
Let $G$ be a $d$-$(u,v)$-UD graph of order $n$, where $f(d)\geq n$. Let $G^*$ be the graph obtained from $G$ by attaching a pendant vertex $u^{\prime}$ to $u$ and a pendant vertex $v^{\prime}$ to $v$. If $E_1(G)>W(G)$, then $E_1(G^*)>W(G^*)$.
\end{theorem}

\begin{proof}
Since $G$ is a UD graph, $G^*$ is also a UD graph in which $u^{\prime}, v^{\prime}$ form a universally diametrical pair. Therefore we have $\varepsilon_{G^*}(u^{\prime})=\varepsilon_{G^*}(v^{\prime})=d+2$ and $\varepsilon_{G^*}(w)=\varepsilon_{G}(w)+1$ for any vertex $w\in V(G)$. Then
\begin{eqnarray*}
E_1(G^*)&=&2(d+2)^2+\sum\limits_{w\in V(G)}(\varepsilon_G(w)+1)^2\\
              &=&E_1(G)+2\varepsilon(G)+n+2(d+2)^2.
 \end{eqnarray*}
 Moreover, from the structure of $G^*$, we have
 \begin{eqnarray*}
\Tr_{G^*}(u^{\prime})&=&d+2+\sum\limits_{w\in V(G)}(d_G(u,w)+1)\\
              &=&\Tr_G(u)+n+d+2.
 \end{eqnarray*}
 Similarly, we have $\Tr_{G^*}(v^{\prime})=\Tr_G(v)+n+d+2$. Note that $\Tr_{G^*}(w)=\Tr_G(w)+d_G(u,w)+d_G(v,w)+2$ for any vertex $w\in V(G)$.  It follows that
 \begin{eqnarray*}
2W(G^*)&=&\Tr_{G^*}(u^{\prime})+\Tr_{G^*}(v^{\prime})+\sum\limits_{w\in V(G)}\Tr_{G^*}(w)\\
              &=&\Tr_G(u)+\Tr_G(v)+2n+2(d+2)+ \\
              & & \sum\limits_{w\in V(G)}\Big(\Tr_G(w)+d_G(u,w)+d_G(v,w)+2\Big)\\
              &=&2[\Tr_G(u)+\Tr_G(v)]+4n+2(d+2)+2W(G),
 \end{eqnarray*}
 that is, $W(G^*)=W(G)+\Tr_G(u)+\Tr_G(v)+2n+d+2$. Note that $\varepsilon_G(u)=\varepsilon_G(v)=d$ for the universally diametrical pair $\{u,v\}$ in $G$. Combining Lemma~\ref{tr-ec} with the assumption that $E_1(G)> W(G)$ and $2d^2+9d+6\geq n$,  we have
 \begin{eqnarray*}
E_1(G^*)-W(G^*)&>&2\varepsilon(G)-\Tr_G(u)-\Tr_G(v)+2(d+2)^2-(d+2)-n\\
              &\geq&2d^2+9d+6-n\\
              &\geq &0,
 \end{eqnarray*}
 finishing the proof of the theorem.
 \end{proof}

In the following we will make use of the {\em eccentric connectivity index}~\cite{SGM1997} of a graph $G$ defined as $\xi ^{c}(G)= \sum_{v \in V(G)}\deg_G(v)\varepsilon_G(v)$, see also~\cite{IG2011,XAD2017, XuLi2016}. The next result is parallel to Theorem~\ref{uni-pend}, but now we compare $E_2$ with $E_1$.

\begin{theorem}\label{uni-pend-e}
Let $G$ be a $d$-$(u,v)$-UD graph of order $n$, size $m\geq n+2d+4$, and $\delta(G)\ge 2$.  If $G^*$ is defined just  as in Theorem~\ref{uni-pend} and  $E_2(G)> E_1(G)$, then $E_2(G^*)>E_1(G^*)$.
\end{theorem}

\begin{proof}
By a similar reasoning as that in the proof of Theorem \ref{uni-pend}, we have \begin{eqnarray*}
E_2(G^*)&=&2(d+2)(d+1)+\sum\limits_{uv\in E(G)}(\varepsilon_G(u)+1)(\varepsilon_G(v)+1)\\
        &=&2(d+2)(d+1)+\sum\limits_{uv\in E(G)}\varepsilon_G(u)\varepsilon_G(v)+\sum\limits_{uv\in E(G)}[\varepsilon_G(u)+\varepsilon_G(v)]+m\\
              &=&2(d+2)(d+1)+E_2(G)+m+\xi^c(G).
 \end{eqnarray*}
Note that  $E_1(G^*)=E_1(G)+2\varepsilon(G)+n+2(d+2)^2$ (see the proof of Theorem \ref{uni-pend}) and $\xi^c(G)\geq 2\varepsilon(G)$ since $\delta(G)\ge 2$. Then the assumptions $E_2(G)>E_1(G)$ and $m\geq n+2d+4$ give
$E_2(G^*)-E_1(G^*) > m-n-2(d+2)+\xi^c(G)-2\varepsilon(G) \geq 0$.
\end{proof}

Theorem~\ref{uni-pend} can be  extended as follows.

 \begin{corollary}\label{pend-path}
 Let $G$ be a $d$-$(u,v)$-UD graph of order $n$ with  $f(d+2\ell-2)\geq n+2\ell-2$. Let $G^{\ell*}$ be the graph obtained from $G$ by attaching a pendant path of length $\ell\geq 1$ to each of $u$ and $v$. If $E_1(G)>W(G)$, then $E_1(G^{\ell*})>W(G^{\ell*})$.
 \end{corollary}

\begin{proof}
Since $G^{1*}\cong G^{*}$, the result for $\ell=1$ follows from Theorem~\ref{uni-pend}. Clearly, $G^{k*}$ is a universally diametrical graph for $k\in [\ell]$. Since $G^{\ell*}$ can be obtained by attaching a pendant vertex to each vertex of universally pair, respectively, in $G^{(\ell-1)*}$ which is order $n+2\ell-2$ and has diameter $d+2\ell-2$, our result holds by repeatedly applying Theorem~\ref{uni-pend}.
\end{proof}

\section{Cartesian product graphs}
\label{sec:Cartesian}

In this final section we prove that  if graphs have the property $W\geq E_1$, then the same property holds for the Cartesian product of these graph. Recall that the {\em Cartesian product} $G\cp H$ of graphs $G$ and $H$ is the graph with $V(G\cp H) = V(G)\times V(H)$ and $(g,h)$ is adjacent to $(g',h')$ if either $gg'\in E(G)$ and $h=h'$, or $g=g'$ and $hh'\in E(H)$. Since $\varepsilon_{G\cp H}(g,h) = \varepsilon_G(g) + \varepsilon_H(h)$ (cf.~\cite{HIK2011}), the following lemma is straightforward.

\begin{lemma}\label{NL1}
If $G$ and $H$ are connected graphs, then
$$E_1(G\cp H) = n(H) E_1(G) + n(G) E_1(H) + 2\varepsilon(G)\varepsilon(H)\,.$$
 \end{lemma}

\begin{theorem}\label{C-product}
If $G$ and $H$ are connected graphs, $W(G)\geq E_1(G)$, $W(H)\geq E_1(H)$, and $\max\{n(G),n(H)\}>2$, then $W(G\cp H) > E_1(G\cp H)$.
\end{theorem}

\begin{proof}
It is well-known for a long time, see~\cite{graovac-1991, yeh-1994}, that $W(G\cp H) = n(H)^2W(G) + n(G)^2W(H)$.  Then, combining Lemma \ref{NL1} with the fact that $E_1(X)\geq \varepsilon(X)$ holds for any connected graph $X$,  and setting $Z = W(G\cp H) - E_1(G\cp H)$, we have
\begin{eqnarray*}
Z & =& n(H)^2W(G) + n(G)^2W(H) - n(H)E_1(G)-n(G)E_1(H) - 2\varepsilon(G)\varepsilon(H) \\
	& \geq & n(H)(n(H)-1)W(G) + n(G)(n(G)-1)W(H)-2\varepsilon(G)\varepsilon(H)  \\
	& \geq & \Big[n(H)(n(H)-1)-\varepsilon(H)\Big]W(G)+\Big[n(G)(n(G)-1)-\varepsilon(G)\Big]W(H)\\
    &>&0\,,
\end{eqnarray*}
where the last inequality holds by the assumption $\max\{n(G),n(H)\}>2$.
\end{proof}

Similarly as Lemma~\ref{NL1}, but with a little more effort, the next result can be deduced.

\begin{lemma}{\rm(\cite{XDG2019})}
Let $G$ and $H$ be two connected graphs. Then
\begin{eqnarray*}
E_2(G\cp H) & = & m(H) E_1(G) + n(H) E_2(G) +m(G) E_1(H) + n(G) E_2(H) + \\
& &  \varepsilon(G)\xi^c(H)+\varepsilon(H)\xi^c(G)\,.
\end{eqnarray*}
\end{lemma}

From Theorems~\ref{thm:2-non-SC} and~\ref{universal}, we know that there exist graphs $G$ satisfying $W(G)\geq \max\{E_1(G),E_2(G)\}$.  Define the \textit{average transmission} of a connected graph $G$ as $\avt(G)=\frac{2W(G)}{n(G)}$. Then we have:

\begin{theorem}
Let $G$ and $H$ be connected graphs with diameters $d_G$ and $d_H$, respectively, and let $W(G) \geq \max\{E_1(G),E_2(G)\}$ and $W(H) \geq \max\{E_1(H),E_2(H)\}$. If $\avt(G)>4d_G^2d_H$ and $\avt(H)>4d_H^2d_G$, then $W(G\cp H) > E_2(G\cp H)$.
\end{theorem}

\begin{proof}
As already mentioned in the proof of Theorem~\ref{C-product}, $W(G\cp H) = n(H)^2W(G) + n(G)^2W(H)$. Since $m(X)\le \binom{n(X)}{2}$  holds for any graph $X$, we have $n(X)^2-m(X)-n(X)\geq m(X)$  for any graph $X$. Hence, setting $A = W(G\cp H) - E_2(G\cp H)$ and using the assumptions $W(G)\geq \max\{E_1(G),E_2(G)\}$ and $W(H)\geq \max\{E_1(H),E_2(H)\}$, we can estimate as follows:
\begin{eqnarray*}
A & =& n(H)^2W(G)-m(H) E_1(G)- n(H) E_2(G)+ n(G)^2W(H)\\
 & & - m(G) E_1(H)- n(G) E_2(H) -\varepsilon(G)\xi^c(H)-\varepsilon(H)\xi^c(G) \\
	& \geq & \Big[n(H)^2-m(H)-n(H)\Big]W(G) + \Big[n(G)^2-m(G)-n(G)\Big]W(H)\\
&&-\varepsilon(G)\xi^c(H)-\varepsilon(H)\xi^c(G)  \\
	& \geq & m(H)W(G)+m(G)W(H)-2m(H)n(G)d_G^2d_H-2m(G)n(H)d_H^2d_G\\
  &=&m(H)\Big[W(G)-2n(G)d_G^2d_H\Big]+m(G)\Big[W(H)-2n(H)d_Gd_H^2\Big]\\
    &>&0.
\end{eqnarray*}
Note that the last inequality holds because of the assumptions $\avt(G)>4d_G^2d_H$ and $\avt(H)>4d_H^2d_G$.
\end{proof}

\section*{Acknowledgements}

K. X. is supported by supported by NNSF of China
(grant No. 11671202) and China-Slovene bilateral grant 12-9,
  K. C. D. is supported by National Research Foundation
 funded by the Korean government (grant no. 2017R1D1A1B03028642). S.K. is supported by Slovenian Research Agency (research core funding P1-0297,
projects J1-9109, J1-1693, N1-0095, and the bilateral grant BI-CN-18-20-008).

\end{document}